\documentclass[12pt,a4paper]{amsart}
\usepackage{graphicx}
\usepackage{color}
\usepackage[colorlinks=true]{hyperref}
\hypersetup{urlcolor=blue, citecolor=blue, draft=false}
\usepackage{subfigure}

\newtheorem{theorem}{Theorem}[section]
\newtheorem{lemma}{Lemma}[section]
\newtheorem{corollary}{Corollary}[section]
\theoremstyle{remark}

\theoremstyle{remark}

\begin{document}
\title[Radius of injectivety for harmonic mappings with fixed analytic part]{Radius of injectivety for harmonic mappings with fixed analytic part}

\author[Jugal Kishore Prajapat and Manivannan Mathi]{Jugal Kishore Prajapat and Manivannan Mathi}

\address{Manivannan Mathi, Department of Mathematics, Central University of Rajasthan, Bandarsindri, Kishangarh-305817, Dist.-Ajmer, Rajasthan, India.}
\email{manivannan.mathi91@gmail.com, jkprajapat@gmail.com}

\begin{abstract}
In this paper, we study non sense-preserving harmonic mappings $f=h+\overline{g}$ in $\mathbb{D}$ when its analytic part $h$ is convex and injective in $\mathbb{D}$ and obtain radius of injectivety.
\end{abstract}


\maketitle
\section{Introduction}
\setcounter{equation}{0}

A complex valued function $f$ is said to be harmonic in a domain $\Omega \subset \mathbb{C}$ if it satisfies $f_{z \overline{z}}(z)=0$ for all $z \in \Omega$. If $\Omega$ is simply connected then such functions can be represented as $f=h+\overline{g}$, \,where  $h$ and $g$ are analytic in $\Omega$. Furthermore, if $g(0)=0$, then this representation is unique.  Let $\mathcal{H}ar(\mathbb{D})$ denote the class of harmonic mappings $f$ in the open unit disk \linebreak $\mathbb{D}=\{z \in\mathbb{C} : |z|<1\}$ with the normalization $h(0)=h'(0)-1=0$ and $g(0)=0$. Such mappings $f$  are uniquely determined by the coefficients of power series
\begin{equation}\label{intro2}
h(z) = z + \sum_{n=2}^{\infty} a_nz^n, \quad \qquad g(z)=\sum_{n=1}^{\infty} b_nz^n  \qquad \qquad (z \in \mathbb{D}).
\end{equation}
Here $h$ is analytic and $g$ is co-analytic part of $f$. The Jacobian $J_f(z) $ \,of \,$f=h+\overline{g}  \in  \mathcal{H}ar(\mathbb{D})$ \,is \,$ J_f(z)   = |h'(z)|^2-|g'(z)|^2$. A function  $f \in  \mathcal{H}ar(\mathbb{D})$ is locally injective in $\mathbb{D}$ if and only if the Jacobian $J_f(z)$ is non-vanishing in $\mathbb{D},$ and sense-preserving if $J_f(z) >0$ in $\mathbb{D}$ (see \cite{lewy}). A harmonic mapping $f$ is said to be close-to-convex if $f(\mathbb{D})$ is close-to-convex, i.e., the complement of $f(\mathbb{D})$ can be written as disjoint union of non-intersecting half lines. The study of harmonic mappings have attracted the attention of complex analysts after Clunie and Sheil-Small \cite{clunie}. For recent results in harmonic  mappings, we refer to \cite{muhanna, Aa, zliu, D.} and the references therein. 
  
Let $\mathcal{H}ol(\mathbb{D})$ denote the class of holomorphic functions $f$ in $\mathbb{D}$ that are normalized by \linebreak $f(0)=f'(0)-1=0$ and $\mathcal{S}$ denote the subclass of  $\mathcal{H}ol(\mathbb{D})$ of injective holomorphic functions in $\mathbb{D}$. Note that  $\mathcal{H}ol(\mathbb{D})  \subset \mathcal{H}ar(\mathbb{D})$. Let  $\mathcal{K}$ denote the class of analytic functions $f \in \mathcal{H}ol(\mathbb{D})$ such that $f(\mathbb{D})$ is convex. It is well known that, convexity of analytic functions in $\mathbb{D}$ is a hereditary property; that is, if $f$ is convex in $\mathbb{D},$ then $f(\mathbb{D}_r)$ is convex for every $r \;(0<r<1),$ where $\mathbb{D}_r=\{z: |z|<r, \;0<r<1\}.$ An analytic function $\mathcal{H}ol(\mathbb{D})$ is said to be starlike function of order $\alpha \;(0 \leq \alpha <1),$ if $\Re(z f'(z)/f(z))> \alpha \; (z \in \mathbb{D})$. Let $\mathcal{B}$ denote the set of all analytic functions $w$ in $\mathbb{D}$ such that $|w(z)|\leq 1$ in $\mathbb{D}.$ A function $w \in \mathcal{B}$ satisfies the inequality  
\begin{equation}\label{11}
|w'(z)| \leq \frac{1-|w(z)|^2}{1-|z|^2}, \qquad z \in \mathbb{D},
\end{equation}
(see \cite[p. 168]{nehari}).  

The analytic parts of harmonic mappings are significant in shaping their geometric properties. For example, if $h$ is convex injective and  $f =h + \overline{g} \in \mathcal{H}ar(\mathbb{D})$ is sense-preserving, then $f(\mathbb{D})$ is close-to-convex \cite{clunie}. In (\cite{kanas1, prajapat}) harmonic mappings $f =h + \overline{g} \in \mathcal{H}ar(\mathbb{D})$ have been studied, where $|g'(0)|= \alpha \in [0,1),\; h$ is convex in one direction in $\mathbb{D}$ and the dilatation $w$ is given by $w(z) = (z + \alpha)/(1 + \alpha z).$ In \cite{bshouty1}, Bshouty {\it et al.} proved the following result of $f=h+\overline{g} \in \mathcal{H}ar(\mathbb{D})$ when $h$ is convex in $\mathbb{D}$.

\begin{lemma}\label{lemma1}
Let  $h$ be analytic and convex in $\mathbb{D}$. Then every harmonic mapping  $f=h+\overline{g}$ where $g'(z)=w(z) h'(z); \;|w(z)|<1$  is close-to-convex in $\mathbb{D}$.
\end{lemma}

Note that, the harmonic mapping in Lemma 1.1 is sense-preserving. In this article, we consider the case of Lemma 1.1 when harmonic mapping $f=h+\overline{g}$\, is not necessarily sense-preserving in $\mathbb{D}$ but satisfies $g(z)=w(z) h(z) \;\;(w \in \mathcal{B})$. We observe that such harmonic mappings are not necessarily sense-preserving and injective in $\mathbb{D}$. For example, the harmonic mapping 
\begin{equation}\label{eq2}
f_1(z)=\frac{z}{1-z}-\overline{\frac{z}{2}}, \qquad z \in \mathbb{D},
\end{equation}
is not sense-preserving in $\mathbb{D}$ as $|g'(-1/2)/h'(-1/2)|=9/8>1$ and not injective in $\mathbb{D}$  (see $Figure \,1$).

\begin{figure*}[h]\label{fig1}
{\resizebox*{5.0cm}{!}{\includegraphics{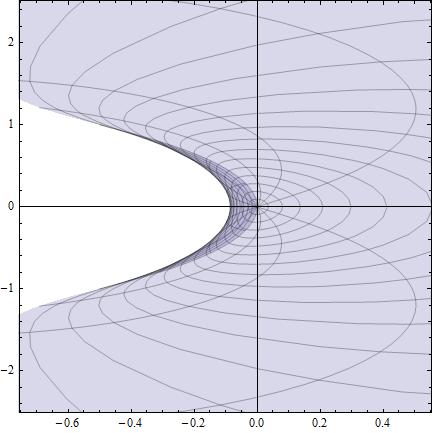}}}
\caption{\label{fig1} The images of $\mathbb{D}$ under $f_1$.}
\end{figure*}

\medskip
Recently, Ponnusamy and Kaliraj \cite{samy} proved the following result for \linebreak $f=h+\overline{g} \in \mathcal{H}ar(\mathbb{D})$ when $h$ is injective in $\mathbb{D}$.

\begin{lemma}\label{lemma2}
Suppose that $f=h+\overline{g} \in \mathcal{H}ar(\mathbb{D})$ is sense-preserving in $\mathbb{D}$ such that $h$ is injective in $\mathbb{D}$. Then the radius of injectivety and close-to-convexity of $f$ is $2-\sqrt{3}.$
\end{lemma}

In this article, we consider the case of Lemma 1.2 when harmonic \linebreak mapping $f=h+\overline{g}$\, is not necessarily sense-preserving in $\mathbb{D}$ but satisfies $g(z)=w(z) h(z) \;(w \in \mathcal{B})$. We observe that such harmonic mappings are not necessarily sense-preserving and injective in $\mathbb{D}$. For example, the harmonic mapping  
\begin{equation}\label{eq4}
f_2(z)=\frac{z}{(1-z)^2}- \overline{\frac{z}{2(1-z)}},  \qquad z\in \mathbb{D},
\end{equation}
is not sense-preserving as $|g'(-1/2)/h'(-1/2)|=3/2>1$ and not injective in $\mathbb{D}$  (see $Figure \;2$).

\begin{figure}\label{fig2}
{\resizebox*{5.0cm}{!}{\includegraphics{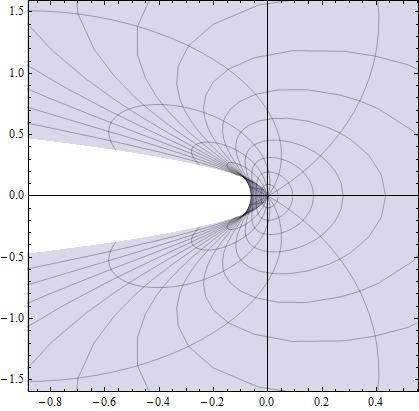}}}\hspace{4pt}
\caption{\label{fig2} The images of $\mathbb{D}$ under $f_2$.}
\end{figure}

\section{Main Results}
To prove our results, we shall  use  the following Lemma. This Lemma was appeared in \cite{chuaqui1} for $\mathbb{D}$, but observing it's proof, we see that this result is valid for all subdisk $\mathbb{D}_r.$ The proof of this special case is so short that we include it here for completeness.

\medskip
\noindent
\begin{lemma}\label{lemma3}
Let  $f=h+\overline{g} $ be a sense-preserving harmonic mapping in \linebreak $\mathbb{D}_r \,(0<r<1)$ and $h$ is injective convex in $\mathbb{D}$. Then $f$ is injective in $\mathbb{D}_r$.

\begin{proof}
Let \;$\Omega =h(\mathbb{D}_r), \; 0<r<1$. Define $\psi:\Omega \longrightarrow \mathbb{C}$ by $\psi(w)=g\circ h^{-1}(w).$ Then $\psi$ is analytic in convex domain $\Omega$ and  $\mathcal{\psi}'(w)= g'(w) /h'(w),$ where  $w=\mathcal{\psi}^{-1}(z)$ and $ |\mathcal{\psi}'(w)|<1.$ Now, let $z_1, z_2 \in \mathbb{D}_r, \;z_1\neq z_2$ such that  $f(z_1)=f(z_2),$ this is equivalent to $h(z_1)-h(z_2)=-(\overline{g(z_1)}-\overline{g(z_2)}).$ Set $w_1=h(z_1)$  and $ w_2=h(z_2)$ so that $w_1-w_2 = -(\overline{g(z_1)}-\overline{g(z_2)}). $ As $h^{-1}(w_1)=z_1$ and $ h^{-1}(w_2)=z_2,$ we have 
\begin{eqnarray*}
w_1-w_2 = -\left( \overline{g(h^{-1}(w_1))} \right) -\left( \overline{g(h^{-1}(w_2))} \right) = \overline{\psi(w_2)}-\overline{\psi(w_1)} .
\end{eqnarray*}
Because $\mathcal{\psi}$ is analytic on the convex domain $\Omega,$ we have $\overline{w_1}-\overline{w_2}=\int_{[w_1,w_2]}\mathcal{\psi}'(w) \, dw,$ which  is not possible as $|\mathcal{\psi}'(w)|<1$ in $\Omega.$ Thus $f(z_1)\neq f(z_2)$. This shows the injectivety of $f$ in $\mathbb{D}_r.$
\end{proof}
\end{lemma}

\medskip
First we prove the following sharp result for harmonic mappings with injective and convex analytic part. 

\noindent
\begin{theorem}\label{thm1}
Let $f=h+\overline{g} \in \mathcal{H}ar(\mathbb{D})$ such that $h$ is injective and convex in $\mathbb{D}$ and $g(z)=w(z)h(z)$, where $w \in \mathcal{B}$. Then $f$ is sense-preserving and injective in $\mathbb{D}_{1 / 3}$. The results is sharp. 
\end{theorem}

\begin{proof} As convexity is hereditary property, $h$ is injective and convex in $D_{1/3}$. Thus, in view of Lemma \ref{lemma3}, it is sufficient to prove that $f$ is sense-preserving in $\mathbb{D}_{1/3}$. A convex function is starlike of order $1/2$ \; \cite[Theorem 2.3.2]{graham}. Hence
$$\frac{z h'(z)}{h(z)}=\frac{1}{1-w(z)},$$
where $w \in \mathcal{B}$. This gives $|h(z)|\leq |z| \left( 1+|z|\right) |h'(z)| \;\; (z \in \mathbb{D}).$ Using this inequality and (\ref{11}), we have
\begin{eqnarray*}
 |g'(z)| & \leq & |w(z)| |h'(z)|+ |h(z)| |w'(z)| \\
 & = &  \left( |w(z)|+\frac{|z|(1-|w(z)|^2)}{1-|z|}  \right) |h'(z)|.
\end{eqnarray*}
If $|z|<\dfrac{1}{3},$ then 
\begin{eqnarray*}
|w(z)|+\frac{|z|(1-|w(z)|^2)}{1-|z|} &< &\frac{1}{2} (2|w(z)|+1-|w(z)|^2)\\
&\leq& \frac{1}{2} (2|w(z)|+2(1-|w(z)|)) =1.
\end{eqnarray*}
Therefore $|g'| < |h'|$, hence $f$ is sense-preserving in $\mathbb{D}_{1/3}.$ To show the sharpness, let
\[h(z)=\frac{z}{1+z}   \quad  {\rm and }  \quad w(z)=\frac{z+\zeta}{1+\zeta z} \]
where $\zeta \in [-1, 1].$ We deduce that $g'(r)= U(r, \zeta) h'(r) $, where
$$U(r, \zeta) =\frac{r+\zeta}{1+r \zeta}+\frac{r(1+r)(1-\zeta^2)}{(1+r \zeta)^2}.$$
Note that $U(r,1) = 1$ and 
$$\left. \frac{\partial }{\partial \zeta} \,U(r, \zeta)\right|_{\zeta=1}=\frac{1-3r}{1+r}.$$
Choose $r$ such that $r \in [1/3, 1)$, then $\left. \frac{\partial }{\partial \zeta} U(r, \zeta)\right|_{\zeta=1} \leq 0$. Hence $U(r,1-\varepsilon)>1$  for each $\epsilon >0$. This gives $ g'(r)> h'(r)> 0,$  for each $r \in [1/3, 1)$. Thus $f$ is not sense-preserving in $|z|<r$ if $r >1/3.$ This complete the proof. 
\end{proof}

\medskip
If dilatation $w$ has the form $w(z)=e^{i \theta} z^n \; \;(\theta \in \mathbb{R}, n \geq 1)$ and $w(z)=c \;\;(c \in \mathbb{C}, |c|<1)$, then we have

\medskip
\noindent
\begin{corollary}\label{cor1}
Let $f=h+\overline{g} \in \mathcal{H}ar(\mathbb{D})$ such that $h$ is injective and convex in $\mathbb{D}$ and $g(z)= e^{i \theta} z^n h(z) \;\;(\theta \in \mathbb{R}, n \geq 1)$, then $f$ is injective in $\mathbb{D}_{r_{n, 1}},$ where $r_{n, 1}$ is the unique root of $\;n \,r^{n+1}+(n+1)r^n-1= 0$ in the interval $(0, 1)$. The constant $r_{n, 1}$ cannot be improved. The constant $r_{n, 2}$ cannot be improved.
\begin{center}    
\begin{tabular}{llllll}
\hline\noalign{\smallskip}
$n$ \quad \vline & 1 & \quad 2 \quad & \quad 3 & \quad 4 & \quad 5\\
\noalign{\smallskip}\hline\noalign{\smallskip}
$r_{n, 1}$ \vline & $\approx 0.414 $ & $\quad 0.5$ & $\;\; \approx 0.5604 $ & $\approx 0.6058 $ & $\approx 0.6415 $    \\
\noalign{\smallskip}\hline
\end{tabular}
\end{center}

\begin{proof}
From the hypothesis $g(z)=e^{i \theta} z^n h(z)$, we obtain  
\begin{eqnarray*}
|g'(z)| &=& \left| n e^{i \theta} z^{n-1} \frac{h(z)}{h'(z)} +e^{i \theta} z^n \right| \\
& \leq & (n |z|^{n+1}+(n+1)|z|^n)|h'(z)|.
\end{eqnarray*}
Hence, $|g'(z)| <|h'(z)|$  \,if \, $n |z|^{n+1}+(n+1)|z|^n \leq 1.$ Thus, $f$ is sense-preserving in $\mathbb{D}_{r_{n, 1}},$ where $r_{n, 1}$ is the unique root of $n \,r^{n+1}+(n+1) r^n-1=0$ in the interval $(0, 1)$. 
\end{proof}
\end{corollary}

\medskip
\noindent
\begin{corollary}\label{cor2}
Let $f=h+\overline{g} \in \mathcal{H}ar(\mathbb{D})$ such that $h$ is injective and convex in $\mathbb{D}$ and $g(z)=c \,h(z) \;\;(c \in \mathbb{C}, |c|<1)$, then $f$ is sense-preserving and close-to-convex in $\mathbb{D}$.

\begin{proof}
Note that $f$ is sense-preserving in $\mathbb{D}$, hence in view of Lemma 1.1, $f$ is close-to-convex in $\mathbb{D}$.
\end{proof}
\end{corollary}

\medskip
Now we prove the following sharp result for harmonic mappings with injective analytic part. 

\medskip
\noindent
\begin{theorem} \label{thm4}
Let $f=h+\overline{g}\in \mathcal{H}ar(\mathbb{D})$ such that $h$ is injective in $\mathbb{D}$ and $g(z) =w(z)h(z)$, where $w \in \mathcal{B}$. Then $f$ is sense-preserving and injective in $\mathbb{D}_{2-\sqrt{3}}$. The result is sharp. 

\begin{proof}
It is well known that the radius of convexity for the class $\mathcal{S}$ is $2-\sqrt{3}$  (see \cite[Theorem 2.2.22]{graham}). Thus, in view of Lemma \ref{lemma3}, it is sufficient to prove that $f$ is sense-preserving in $\mathbb{D}_{2-\sqrt{3}}.$ For $h \in \mathcal{S}$, we have 
$$|h(z)|\leq \frac{|z|(1+|z|)}{1-|z|}|h'(z)|,  \qquad z \in \mathbb{D},$$
(see\, \cite[Theorem 1.1.6]{graham}). Using this inequality and (\ref{11}), we have
\begin{equation}
|g'(z)|\leq \left( |w(z)|+\frac{|z|(1-|w(z)|^2)}{(1-|z|)^2}\right)|h'(z)|.
\end{equation}
If $|z|^2-4|z|+1>0,$ then 
\[|w(z)|+\frac{  |z|\left( 1-|w(z)|^2\right) }{(1-|z|)^2} < 1,\] 
hence $|g'|<|h'|$ in $\mathbb{D}$. Therefore, $f$ is sense-preserving in a disk $\mathbb{D}_r$, where $r$ is unique root of $r^2-4r+1=0$ in the interval  $(0, 1)$. This shows that $f$ is sense-preserving  in $\mathbb{D}_{2-\sqrt{3}}$. To show the sharpness, let
\[h(z)=\frac{z}{(1+z)^2}   \quad {\rm and}  \quad w(z)=\frac{z+\zeta}{1+\zeta z}, \]
where $\zeta \in [-1, 1].$ A computation gives $g'(r)=V(r, \zeta) h'(r)$, where
\begin{eqnarray*}
V(r, \zeta)= \frac{r+\zeta}{1+r \zeta}+\frac{1-\zeta^2}{(1+r \zeta)^2}\frac{r(1+r)}{1-r}.
\end{eqnarray*}
Note that $V(r,1) = 1$ and 
\[\left. \frac{\partial }{\partial \zeta} \,V(r, \zeta) \right|_{\zeta=1}=\frac{1-4r+r^2}{1-r^2}.\]
Choose $r$ such that $r \in [2-\sqrt{3}, 1),$ then  $\left. \frac{\partial }{\partial \zeta} \,V(r, \zeta) \right|_{\zeta=1} <0.$ Therefore $V(r,1-\epsilon)>1$ for $\epsilon >0$. This shows that $ g'(r)> h'(r)> 0$ for $r \in (2-\sqrt{3}, 1)$. Thus $f$ is not injective in $|z|<r$ if $r >2-\sqrt{3}.$ This complete the proof. 
\end{proof}
\end{theorem}

If dilatation $w$ has the form $w(z)= e^{i \theta} z^n \; (\theta \in \mathbb{R}, n \geq 1)$, then we have

\medskip
\noindent
\begin{corollary}\label{cor3}
Let $f =h+\overline{g} \in \mathcal{H}ar(\mathbb{D})$ such that $h$ is injective in $\mathbb{D}$ and \linebreak $g(z)= e^{i \theta} z^n h(z) \;(\theta \in \mathbb{R}, n \geq 1)$, then $f$ is injective in $\mathbb{D}_{r_{n,2}},$ where $r_{n, 2}$ is the unique root of $(n-1)r^{n+1}+(n+1)r^n+r-1= 0$ in the interval $(0, 1)$. The constant $r_{n, 2}$ cannot be improved.
\begin{center}    
\begin{tabular}{llllll}
\hline\noalign{\smallskip}
$n$ \quad \vline & 1 & \quad 2 \quad & \quad 3 & \quad 4 & \quad 5\\
\noalign{\smallskip}\hline\noalign{\smallskip}
$r_{n, 2}$ \vline & $\approx 0.333 $ & $\approx 0.414$ & $\;\; \approx 0.4738$ & $\approx 0.5201 $ & $\approx 0.5574 $    \\
\noalign{\smallskip}\hline
\end{tabular}
\end{center}

\begin{proof}
We have 
\begin{eqnarray*}
|g'(z)| &=& \left| n e^{i \theta} z^{n-1} \frac{h(z)}{h'(z)} +e^{i \theta} z^n \right| \\
& \leq & \left( n |z|^{n-1} \frac{|z|(1+|z|)}{1-|z|}+|z|^n \right) |h'(z)|.
\end{eqnarray*}
Hence, $|g'(z)| <|h'(z)|$  \,if $(n-1)|z|^{n+1}+(n+1)|z|^n+|z|-1 \geq 0.$ Thus, $f$ is sense-preserving in $\mathbb{D}_{r_{n, 2}},$ where $r_{n, 2}$ is the unique root of $(n-1) r^{n+1}+(n+1) r^n+r-1=0$ in the interval $(0, 1)$.  
\end{proof}

\end{corollary}

\end{document}